\documentclass[11pt]{amsart} 

\usepackage[margin=1.4in]{geometry} 
\usepackage{amsmath,enumerate,hyperref,amsthm,amssymb,bbm,mathrsfs} 
\usepackage{pgfplots}

\theoremstyle{theorem} 
\newtheorem{theorem}{Theorem}
\newtheorem{proposition}{Proposition}

\newtheorem{lemma}{Lemma}
\theoremstyle{definition} 
 
\newtheorem{defn}{Definition}

\newcommand{\calh}{\mathcal{H}}
\newcommand{\calf}{\mathcal{F}}
\newcommand{\cala}{\mathcal{A}}

\newcommand{\ep}{\varepsilon}

\newcommand{\bbr}{\mathbb{R}}

\newcommand{\lip}{\langle}
\newcommand{\rip}{\rangle}

\DeclareMathOperator{\re}{Re}
\DeclareMathOperator{\im}{Im}
\DeclareMathOperator{\supp}{supp}
\newcommand{\ind}{\mathbbm{1}}

\usepackage{scalerel,stackengine}
\stackMath
\newcommand\reallywidehat[1]{%
\savestack{\tmpbox}{\stretchto{%
  \scaleto{%
    \scalerel*[\widthof{\ensuremath{#1}}]{\kern-.6pt\bigwedge\kern-.6pt}%
    {\rule[-\textheight/2]{1ex}{\textheight}}
  }{\textheight}%
}{0.5ex}}%
\stackon[1pt]{#1}{\tmpbox}%
}

\title[Decay of Fractional Wave Equation with Dense Damping]{On the Energy Decay Rate of the Fractional Wave Equation on $\bbr$ with Relatively Dense Damping}
\author{Walton Green}
\address{School of Mathematical and Statistical Sciences, Clemson University, Clemson, South Carolina 29634}
\email{awgreen@clemson.edu}

\newcommand{\ds}{(-\partial_{xx}+1)^{s/2}}
\newcommand{\dss}{(-\partial_{xx}+1)^{s/4}}

\newcommand{\dssi}{(-\partial_{xx}+1)^{-s/4}}

\begin{document}

\begin{abstract}
We establish upper bounds for the decay rate of the energy of the damped fractional wave equation when the averages of the damping coefficient on all intervals of a fixed length are bounded below. If the power of the fractional Laplacian, $s$, is between 0 and 2, the decay is polynomial. For $s \ge 2$, the decay is exponential. Our assumption is also necessary for energy decay. 
Second, we prove that exponential decay cannot hold for $s<2$ if the damping vanishes at all.
\end{abstract}

\maketitle

Consider the following damped fractional wave equation on $\bbr$ for $s>0$ and $\gamma : \bbr \to \bbr_{\ge 0}$:
	\begin{equation}\label{eq:1} w_{tt}(x,t) +\gamma(x)w_t(x,t) + \ds w(x,t)=0, \quad (x,t) \in \bbr \times \bbr_+.\end{equation}
The damping force is represented by $\gamma w_t$. Herein, we study the decay rate of the energy of $w$, defined by
	\[ E(t) = \|(w(t),w_t(t))\|_{H^{s/2} \times L^2} := \left(\int_\bbr |\dss w(x,t)|^2 + |w_t(x,t)|^2 \, dx \right)^{1/2}\]
Standard analysis shows that if $\gamma=0$, then the energy is conserved, i.e. there is no decay. On the other hand, for constant damping $\gamma=c >0$, it can be shown that $E(t)$ decays exponentially. Thus, the interest in this problem is in interpolating between these cases.
The fractional model (\ref{eq:1}) was introduced recently by Malhi and Stanislavova in \cite{malhi19}. Our results are inspired by their paper, but we are even able to recover what is known in the classical case of $s=2$ from a new perspective.

In this classical setting, the problem was initially studied on bounded domains under the so-called Geometric Control Condition (GCC) \cite{b-l-r,rauch74} which requires $\gamma$ to be positive (in some sense) on certain geodesic curves determined by the geometry of the domain. On $\bbr$, or more generally $\bbr^d$, 
the GCC simplifies to the following: there exist $R$ and $c>0$ such that for all line segments $\ell \in \bbr^d$ of length $R$, 
	\begin{equation}\label{eq:gcc} \int_\ell \gamma(x) \, dx >c .\end{equation}
Extending the results of Bardos, Lebeau, Rauch, Taylor, and Phillips \cite{b-l-r,rauch74} to the non-compact case, Burq and Joly \cite{burq16} proved exponential decay of the energy under the GCC on $\bbr^d$ using the semiclassical analysis of \cite{zworski-book}. These methods, utilizing pseudodifferential calculus, require $\gamma$ to be sufficiently smooth.

For the one-dimensional problem, this smoothness condition was relaxed by Malhi and Stanislavova in \cite{malhi18} to $\gamma$ which are continuous and bounded. Moreover, through intricate spectral analysis, it is shown that condition (\ref{eq:gcc}) is equivalent to exponential decay of the energy of solutions to (\ref{eq:1}). 
One special case of our results (Theorem \ref{thm:1} with $s=2$) 
recovers this result and moreover our proof of this equivalence is of a different nature, applying ideas from the study of the harmonic analyst's uncertainty principle, specifically the Paneah-Logvinenko-Sereda Theorem (see Theorem \ref{thm:kov-pls} below).

The initial investigation into the fractional case in \cite{malhi19} 
required the restrictive assumption that the set
$\{x \in \bbr : \gamma(x) \ge \ep\}$ contains a periodic set. 
In this case, it was proved that the rate of decay is polynomial if $s <2$ and exponential if $s \ge 2$. 
Herein, we relax this periodic condition on $\{\gamma \ge \ep\}$ to require that 
$\{\gamma \ge \ep\}$ be relatively dense \cite{havin12} which means there exists $R>0$ such that
	\begin{equation}\label{eq:gamma-rel-dense} \inf_{a \in \bbr} m(\{\gamma \ge \ep\} \cap [a-R,a+R]) >0 \end{equation}
where $m$ is the Lebesgue measure. 

\begin{theorem}\label{thm:1}
Let $0 \le \gamma \in L^\infty(\bbr)$. There exists $R>0$ such that
	\begin{equation}\label{eq:gamma-dense} \inf_{a \in \bbr} \int_{a-R}^{a+R} \gamma(x) \, dx >0 \end{equation}
if and only if there exists $C,\omega>0$ such that
	\[ E(t) \le \left\{ \begin{array}{lrl} C(1+t)^{\tfrac{-s}{4-2s}}\|w(0),w_t(0)\|_{H^s \times H^{s/2}} & \mbox{if}& 0 < s< 2\\[2mm] Ce^{-\omega t}E(0) & \mbox{if}& s \ge 2 \end{array} \right. \]
for all $t>0$ whenever the right-hand side is finite.
\end{theorem}

%
Note that for $\gamma$ bounded, the condition (\ref{eq:gamma-dense}) is equivalent to $\{ x \in \bbr : \gamma(x) \ge \ep\}$ being a relatively dense set (\ref{eq:gamma-rel-dense}) for $\ep$ small enough. However, if $\gamma$ is unbounded, then (\ref{eq:gamma-dense}) is the weaker condition.

The above result does not say anything about the optimality of the rates. However, we can answer the question posed in \cite{malhi19} concerning the value of the threshold between exponential and polynomial decay. In the final section, we show that exponential decay neccesitates that $s$ be greater than $2$ (as long as $\gamma$ is not bounded away from zero), thus establishing $s=2$ as the threshold.

\begin{theorem}\label{thm:2}
Let $0 \le \gamma \in L^\infty(\bbr)$ and $s >0$. Suppose
	\begin{itemize}
		\item[(i)] $m(\{\gamma =0\}) >0$.
		\item[(ii)] There exists $C,\omega >0$ such that 
	\[ E(t) \le Ce^{-\omega t} E(0) \]
for all $t>0$ and $E(0)<\infty$. 
	\end{itemize}
Then $s \ge 2$.
\end{theorem}

The main ingredient in our proof is a resolvent estimate for the fractional Laplacian (Proposition \ref{prop:res} below). In proving this, we will rely on the study of the uncertainty principle for the Fourier transform \cite{havin12} which is defined by
	\[ \calf(f)(\xi) := \hat f (\xi) = \dfrac{1}{\sqrt{2\pi}}\int_\bbr f(x) e^{-ix\xi} \, dx \]
for $\xi \in \bbr$, $f \in L^1(\bbr) \cap L^2(\bbr)$. $\calf$ then uniquely extends to a unitary operator on $L^2(\bbr)$. The manifestation of the uncertainty principle we will use is a generalization due to O. Kovrijkine of the classical Paneah-Logvinenko-Sereda Theorem \cite{paneah61,logvinenko74}.
\begin{theorem}[Thm 2 from \cite{kovrijkine01}]\label{thm:kov-pls}
Let $\{J_k\}_{k=1}^n$ be intervals in $\bbr$ with $|J_k|=b$. Let $E \subset \bbr$ which is relatively dense. Then, there exists $c>0$ such that
	\[ \|f\|_{L^p(E)} \ge c \| f\|_{L^p(\bbr)} \]
for all $f \in L^p, p \in [1,\infty]$ with $\supp \hat f \subset \bigcup_{k=1}^n J_k$. Moreover, $c$
depends only on the number and size of the intervals, not on how they are placed.
\end{theorem}

In the proof of the proposition, we will only need the case when $p=2$ and there are two intervals $J_1,J_2$.

In order to conclude the polynomial or exponential decay in Theorem \ref{thm:1}, we will use the following two results on semigroups which connect resolvent bounds for the generator to the decay of the semigroup. For exponential decay, there is the following characterization from \cite[Theorem 3]{huang85} (See also \cite{gearhart78,pruss84}).
\begin{theorem}[Gearhart-Pruss Test]\label{thm:pruss}
Let $e^{tA}$ be a $C_0$-semigroup in a Hilbert space $\calh$ and assume there exists $M>0$ such that $\|e^{tA}\| \le M$ for all $t \ge 0$. Then, there exists $C,\omega>0$ such that
	\[ \|e^{tA}\| \le Ce^{-\omega t} \]
if and only if $i\bbr \subset \rho(A)$ and $\sup_{\lambda \in \bbr} \|(A-i\lambda)^{-1}\| < \infty$.
\end{theorem}

For the polynomial decay, we use the following result from \cite[Theorem 2.4]{borichev10}:
\begin{theorem}[Borichev-Tomilov]\label{thm:borichev}
Let $e^{tA}$ be a $C_0$-semigroup on a Hilbert space $\calh$. Assume there exists $M>0$ such that $\|e^{tA}\| \le M$ for all $t \ge 0$ and $i\bbr \subset \rho(A)$. Then for a fixed $\alpha > 0$, 
	\[ \|e^{tA}A^{-1}\|=O(t^{-1/\alpha}) \mbox{ as } t\to \infty \]
if and only if $\|(A-i\lambda)^{-1}\|=O(\lambda^\alpha)$ as $\lambda \to\infty$.
\end{theorem}

\section{Resolvent Estimates}

\begin{proposition}\label{prop:res}
Let $\Omega \subset \bbr$ be relatively dense, $s>0$. There exists $c>0$ (depending on $\Omega,s$) such that for all $f \in L^2(\bbr)$, $\lambda \ge 0$.
	\begin{equation}\label{eq:res} c\|f\|_{L^2(\bbr)}^2 \le (1+\lambda)^{\tfrac 2s-2}\|(\ds-\lambda)f\|^2_{L^2(\bbr)} + \|f\|_{L^2(\Omega)}^2. \end{equation}
\end{proposition}
The operator $\ds$ is understood as a strictly positive Fourier multiplier:
	\[ \ds f(x) := \dfrac{1}{\sqrt{2\pi}}\int_{\bbr} (|\xi|^2+1)^{s/2}\hat f(\xi) e^{ix\xi} \, d\xi. \]
Throughout, we denote by $\|\cdot\|$ the norm $\|\cdot \|_{L^2(\bbr)}$. We begin with the following algebraic lemma.
\begin{lemma}\label{lemma:1}
Let $s>0$. There exists $c_s>0$ such that
	\[ |\tau^s -\lambda| \ge c_s(1+\lambda)^{1-1/s} \]
for all $\tau,\lambda \ge 0$ in the region $| \tau-\lambda^{1/s} | > 1$.
\end{lemma}

\begin{proof}
First, for any $s>0$, there exists $d_s>0$ such that
	\[d_s \max(x,y)^{s-1} |x-y| \le |x^s -y^s| 
	\]
for all $x,y \in \bbr_+$. Next, consider two cases. 
\begin{itemize}
	\item[(i)] If $\tau \ge \lambda^{1/s}+1$, then
	\[ |\tau^s -\lambda| \ge d_s\max(\tau,\lambda^{1/s})^{s-1} | \tau-\lambda^{1/s} | =  d_s\tau^{s-1} | \tau-\lambda^{1/s} |.\]
The function $x \mapsto x^{s-1}(x-\mu)$ is positive and increasing for $x >\mu+1$, so we can bound the final term from below by its value at $\tau=\lambda^{1/s}+1$ which yields 
 	\[|\tau^s -\lambda| \ge d_s (\lambda^{1/s}+1)^{s-1} .\]

\item[(ii)]
If $\tau \le \lambda^{1/s}-1$, then 
	\[ |\tau^s -\lambda| \ge d_s\max(\tau,\lambda^{1/s})^{s-1} \cdot 1 = d_s (\lambda^{1/s})^{s-1}.  \]
If $s <1$, then $s-1 <0$ so $(\lambda^{1/s})^{s-1} \ge (\lambda^{1/s}+1)^{s-1}$. Since $0 \le \tau \le \lambda^{1/s}-1$, $\lambda \ge 1$. So, for $s \ge 1$,
	\[ (\lambda^{1/s})^{s-1} = \dfrac{(\lambda^{1/s}+\lambda^{1/s})^{s-1}}{2^{s-1}} \ge\dfrac{(\lambda^{1/s}+1)^{s-1}}{2^{s-1}}. \]
\end{itemize}
Therefore, there exists $c_s$ such that
	\[ |\tau^s -\lambda| \ge c_s (\lambda^{1/s}+1)^{s-1} \ge c_s(\lambda+1)^{1-\tfrac 1s}\]
where in the final step, we have used the fact that for $p \le q$, $(x^q+y^q)^{1/q} \le (x^p+y^p)^{1/p}$.
\end{proof}

\begin{proof}[Proof of Proposition \ref{prop:res}.]
Let $\lambda \ge 0$, $s>0$ and $g \in L^2(\bbr)$ such that $\supp \hat g \subset A_\lambda := \{ \xi \in \bbr : \big| (|\xi|^2+1)^{1/2}-\lambda^{1/s} \big| \le 1 \}$. Notice that $A_\lambda$ is the union of two intervals of length no more than $4$. Therefore, for $\Omega \subset \bbr$ which is relatively dense, by Theorem \ref{thm:kov-pls}, there exists $C>0$ (independent of $\lambda$ and $g$) such that
	\[ \|g\|
	\le C \|g\|_{L^2(\Omega)}. \]
Denote by $P_\lambda$ the projection $P_\lambda f = \calf^{-1}(\ind_{A_{\lambda}} \calf(f))$. Then, for $f \in L^2(\bbr)$,
	\begin{align*} \|f\|^2 &= \| P_\lambda f\|^2 + \|(I-P_\lambda)f\|^2 \\
				&\le C\|P_\lambda f\|_{L^2(\Omega)}^2 + \|(I-P_\lambda)f\|^2 \\
				&= C\|f-(I-P_\lambda)f\|_{L^2(\Omega)}^2 + \|(I-P_\lambda)f\|^2 \\
				&\le 2C\|f\|_{L^2(\Omega)}^2 + 2C\|(I-P_\lambda)f\|^2_{L^2(\Omega)}+\|(I-P_\lambda)f\|^2 \\
				&\le 2C\|f\|_{L^2(\Omega)}^2 + (2C+1)\|(I-P_\lambda)f\|^2 .
	\end{align*}
It remains to estimate the final term. Applying Lemma \ref{lemma:1} with $\tau = (|\xi|^2+1)^{1/2}$, we obtain
	\begin{align*}
		\|(\ds-\lambda)f\|^2
			&= \int [(|\xi|^2+1)^{s/2}-\lambda]^2|\hat f(\xi)|^2 \, d\xi \\
			&\ge \int_{A_\lambda^c} [(|\xi|^2+1)^{s/2}-\lambda]^2|\hat f(\xi)|^2 \, d\xi \\
			&\ge c_s(\lambda+1)^{2-\tfrac{2}{s}}\int_{A_\lambda^c} |\hat f(\xi)|^2 \, d\xi \\
			&=c_s(\lambda+1)^{2-\tfrac 2s}\|(I-P_\lambda)f\|^2.
	\end{align*}

\end{proof}

To apply (\ref{eq:res}) to the wave equation (\ref{eq:1}), we first represent the wave equation as a semigroup: Setting $W(t) = (w(t),w_t(t))$, we see that (\ref{eq:1}) is equivalent to
	\[ \dfrac d{dt}W(t) = \cala_\gamma W(t) \]
where $\cala_\gamma : H^{s} \times H^{s/2} \to H^{s/2} \times L^2$ is densely defined by $A_\gamma(u_1,u_2) = (u_2,-\ds u_1 - \gamma u_2)$. 
The Sobolev space $H^r$ for $r>0$ is defined by the decay of the Fourier transform:
	\[ H^r := \left\{ u \in L^2 : \|u\|_{H^r}^2 = \int_\bbr (|\xi|^2+1)^r |\hat u(\xi)|^2 \,d\xi < \infty \right\}. \]
The definition above is more convenient for our setting so that $\|u\|_{H^{s/2}} = \|\dss u\|$
, but the multiplier is equivalent to the usual multiplier $(|\xi|+1)^{2r}$.
It is standard that $\cala_0$ is a closed skew-adjoint operator therefore $e^{t\cala_0}$ is a semigroup of unitary operators. Then, since $\gamma \ge 0$, for 
$U=(u_1,u_2) \in H^{s} \times H^{s/2}$,
	\[ \re \lip \cala_\gamma^*U,U\rip_{H^{s/2} \times L^2} = \re \lip \cala_\gamma U,U\rip_{H^{s/2} \times L^2} \]
	\[ = \re \lip \cala_0 U,U \rip_{H^{s/2} \times L^2} - \lip \gamma u_2,u_2 \rip_{L^2} = - \lip \gamma u_2,u_2 \rip_{L^2} \le 0. \]
Moreover, since $\gamma \in L^\infty(\bbr)$, the domain of $\cala_\gamma$ is the same as $\cala_0$. So, by classical semigroup theory \cite{pazy83} $e^{tA_\gamma}$ is a $C_0$-semigroup of contractions. We now apply Proposition \ref{prop:res} to $\cala_0$ and $\cala_\gamma$. The first step is an observability inequality for the undamped wave equation (\ref{eq:1}).

\begin{proposition}\label{prop:res-wave}
Let $\Omega \subset \bbr$ be relatively dense, $s> 0$. Then, there exists $c>0$ such that
	\[ c\|U\|^2_{H^{s/2} \times L^2} \le (|\lambda|+1)^{\tfrac 4s -2}\|(\cala_0-i\lambda)U\|^2_{H^{s/2} \times L^2} + \|u_2\|_{L^2(\Omega)}^2 \]
for all $U=(u_1,u_2) \in H^{s} \times H^{s/2}$ and $\lambda \in \bbr$.
\end{proposition}
\begin{proof}
For $U=(u_1,u_2) \in H^{s}(\bbr) \times H^{s/2}(\bbr)$, set $w_1= \dss u_1-iu_2$ and $w_2=\dss u_1+iu_2$. First, by the parallelogram identity,
	\[ \|w_1\|_{L^2(\bbr)}^2 + \|w_2\|_{L^2(\bbr)}^2 = 2\|\dss u_1\|^2 + 2\|u_2\|^2 = 2\|U\|^2_{H^{s/2} \times L^2}.\]
Second,
	\begin{align*}
		\| (\cala_0-\lambda I)&U\|_{H^{s/2} \times L^2}^2 = \|\dss(-\lambda u_1+ u_2)\|^2 + \|-\ds u_1 -\lambda u_2 \|^2 \\
			&= \|-\lambda\dfrac{w_1+w_2}{2}+i\dss\dfrac{w_1-w_2}{2} \|^2 \\
			&\hspace{5ex}+ \|-\dss\dfrac{w_1+w_2}{2} -i\lambda\dfrac{w_1-w_2}{2} \|^2 \\
			&= \|-i\lambda w_1 - \dss w_1\|^2 + \|-i\lambda w_2+\dss w_2\|^2.
	\end{align*}

So, applying Proposition \ref{prop:res} to $w_1$ with $s$ replaced by $s/2$, we have, for $\lambda \ge 0 $,
	\begin{align}
		2c\|U&\|^2_{H^{s/2}\times L^2} = c(\|w_1\|^2 + \|w_2\|^2) \notag\\
			&\le (|\lambda|+1)^{\tfrac 4s-2} \|(\dss-\lambda)w_1\|^2 + \|w_1\|_{L^2(\Omega)}^2 + c\|w_2\|^2\notag\\
			&\le (|\lambda|+1)^{\tfrac 4s-2} \|(\dss-\lambda)w_1\|^2 + 2\|w_1-w_2\|_{L^2(\Omega)}^2 + (c+2)\|w_2\|^2\notag\\
			&\le (|\lambda|+1)^{\tfrac 4s-2} \|(\dss-\lambda)w_1\|^2 + 8\|u_2\|_{L^2(\Omega)}^2 + \dfrac{c+2}{(|\lambda|+1)^{2}}\|(\dss +\lambda) w_2\|^2\notag\\
			&\le (c+2)(|\lambda|+1)^{\tfrac 4s-2} \| (\cala_0-i\lambda I)U\|_{H^{s/2} \times L^2}^2 + 8\|u_2\|_{L^2(\Omega)}^2. \notag 
	\end{align}
We get the case $\lambda <0$ by exchanging the roles of $w_1$ and $w_2$. 
\end{proof}

Finally we extend this to $\cala_\gamma-i\lambda I$ and prove Theorem \ref{thm:1}.

\section{Proof of the Decay Rates in Theorem \ref{thm:1}}
First notice that for any $R,\ep >0$, $a \in \bbr$,
	\[ \int_{a-R}^{a+R} \gamma(x) \, dx \le \|\gamma\|_\infty m(\{\gamma \ge \ep\} \cap [a-R,a+R])+2R \ep .\]
So, (\ref{eq:gamma-dense}) implies that $\{\gamma > \ep\}$ is relatively dense for $\ep$ small enough. Therefore, taking $\Omega=\{\gamma \ge \ep\}$ and applying Proposition \ref{prop:res-wave},
	\begin{align}\notag c\|U\|^2_{H^{s/2} \times L^2} &\le (|\lambda|+1)^{\tfrac 4s-2}\| (\cala-i\lambda I)U\|_{H^{s/2} \times L^2}^2 + \|u_2\|_{L^2(\Omega)}^2 \\
			\label{eq:5}&\le 2(|\lambda|+1)^{\tfrac 4s-2} \| (\cala_\gamma-i\lambda I)U\|_{H^{s/2} \times L^2}^2 + \left[2(|\lambda|+1)^{\tfrac 4s-2} + \ep^{-2}\right]\|\gamma u_2\|_{L^2(\Omega)}^2. 
	\end{align}
We estimate the final term. Since $\cala_0$ is skew-adjoint,
	\[ \re \lip (\cala_\gamma-i\lambda I)U,U \rip = \re \lip (\cala_0-i\lambda I)U,U \rip - \lip \gamma u_2,u_2 \rip = -\|\sqrt{\gamma} u_2\|^2 \]
which implies \[ D\|\gamma u_2\|^2 \le D\|\gamma\|_\infty \|\sqrt \gamma u_2\|^2 \le \dfrac{D^2\|\gamma\|_\infty^2\|(\cala_\gamma-i\lambda)U\|^2}{\delta} +  \delta\|U\|^2 \]
for any $D,\delta>0$. Choosing $D=2(|\lambda|+1)^{\tfrac 4s-2} + \ep^{-2}$ and $\delta = c/2$, from (\ref{eq:5}) we obtain
	\[ c\|U\|^2_{H^{s/2} \times L^2} \le C\left[ (|\lambda|+1)^{\tfrac 4s-2} + (|\lambda|+1)^{\tfrac 8s-4} + 1\right] \|(\cala_\gamma -i\lambda I)U\|_{H^{s/2} \times L^2}^2 + \dfrac c2\|U\|_{H^{s/2} \times L^2}^2.\]
Thus, we have proved the following estimate for $(\cala_\gamma -i\lambda I)^{-1}$:
	\begin{equation}\label{eq:4} \|(\cala_\gamma -i\lambda I)^{-1}\|_{H^{s/2} \times L^2 \to H^{s/2} \times L^2} \le \left\{ \begin{array}{ll} C(|\lambda|+1)^{\tfrac 4s-2} & 0 < s <2 \\ C & s \ge 2 .\end{array} \right. \end{equation}
Applying the Theorems \ref{thm:pruss} and \ref{thm:borichev} allows one to conclude the decay rates in Theorem \ref{thm:1} from (\ref{eq:4}).

\section{Neccessity of (\ref{eq:gamma-dense}) and Threshold Value}
In this final section we prove the converse in Theorem \ref{thm:1} and subsequently Theorem \ref{thm:2}. 
By the Gearhart-Pruss Test (Theorem \ref{thm:pruss}) and Borichev-Tomilov (Theorem \ref{thm:borichev}), the decay rates of the energy in Theorem \ref{thm:1} imply
	\begin{equation}\label{eq:3.1} c\|U\|_{H^{s/2} \times L^2}^2 \le \|(\cala_\gamma-i\lambda I)U\|_{H^{s/2} \times L^2}^2 \end{equation}
for some $c=c(s,\lambda)>0$ and for all $U \in H^{s/2} \times L^2$ and all $\lambda \in \bbr$. Taking $U=(\dssi u,iu)$ for $u \in L^2(\bbr)$, we have
	\[ 2c\|u\|^2 \le \|(-\lambda +\dss)u\|^2+\|(-\dss -i\gamma+\lambda)u\|^2 \]
	\begin{equation}\label{eq:3.2}\le 3\|(\dss-\lambda)u\|^2 +2\|\gamma u\|^2. \end{equation}
\subsection{Converse in Theorem \ref{thm:1}}
Now, we only consider the special case $\lambda=1$. Let $u \in L^2(\bbr)$ such that $\supp \hat u \subset [-D,D]$ for some $D>0$ to be fixed later. For such $u$,
		\[ \|(\dss-1)u\|^2 = \int_{-D}^D[(|\xi|^2+1)^{s/4}-1]^2|\hat u(\xi)|^2 \, d\xi \le [(D^2+1)^{s/4}-1]^2\|u\|^2. \]
So, taking $D$ small enough, we obtain that there exists $C>0$ such that
	\begin{equation}\label{eq:pls} \|u\|^2 \le C\|\gamma u\|^2 \end{equation}
for all $u \in L^2(\bbr)$ satisfying $ \supp \hat u \subset [-D,D]$. Set $f(x) = \tfrac{\sin(Dx)}{Dx}$. Then, $\supp \hat f \subset [-D,D]$. For each $a \in \bbr$, set $f_a(x) = f(x-a)$. Of course, $\supp \hat f_a \subset [-D,D]$ and $\|f_a\|=\|f\|$. Thus, for any $R>0$,
	\[ \|f\|^2 = \|f_a\|^2 \le C\|\gamma f_a\|^2 = C\int_{[a-R,a+R]}+\int_{[a-R,a+R]^c} |\gamma(x) f_a(x)|^2 \, dx \]
The second integral goes to 0 (uniformly in $a$) as $R \to \infty$ since $\gamma$ is bounded and $f \in L^2$. The first integral becomes
	\[ \int_{a-R}^{a+R} |\gamma(x) f_a(x)|^2 \, dx \le \|\gamma\|_\infty \int_{a-R}^{a+R} \gamma(x) \, dx \]
since $f$ is bounded by $1$. Thus there exists $R$ large such that (\ref{eq:gamma-dense}) holds.

We remark that to prove the neccessity of the condition (\ref{eq:gamma-dense}), the decay rates from Theorem \ref{thm:1} can be replaced by an a priori weaker condition, namely that there exists $\lambda \ge 1$ such that $i \lambda \in \rho(\cala_\gamma)$ and $\cala_\gamma-i\lambda$ has closed range. Then, setting $\mu = \sqrt{\lambda^{2/s}-1}$, we obtain (\ref{eq:pls}) for $\supp \hat u \subset [\mu-D,\mu+D]$ ($D$ small enough). The proof is completed analogously by taking $f(x) = e^{i\mu x}\tfrac{\sin(Dx)}{Dx}$.

\subsection{Proof of Theorem \ref{thm:2}}
Now, to prove the threshold value (Theorem \ref{thm:2}), we use the fact that exponential decay yields (\ref{eq:3.1}) with $c$ independent of $\lambda$, from which (\ref{eq:3.2}) follows. Suppose that $s <2$. We will derive a contradiction. In this case, we take $\supp \hat u \subset \{ \xi \in \bbr : \bigr| (|\xi|^2+1)^{s/4} -\lambda \bigr| \le K\} =: A_\lambda(K)$ for $K$ to be chosen later. Then, we have
	\[ \|(\dss-\lambda)u\|^2 = \int_{A_\lambda(K)} [(|\xi|^2+1)^{s/4}-\lambda ]^2 |\hat u(\xi)|^2 \, d\xi \le K^2\|u\|^2.\]
So taking $K$ small enough, we have, as above,
	\begin{equation}\label{eq:uncp} c\|u\| \le \|\gamma u\| \end{equation}
whenever $\supp \hat u \subset A_\lambda(K)$, $\lambda \in \bbr$. $A_\lambda(K)$ is the union of the two intervals
	\[ \pm\left[\sqrt{(\lambda-K)^{4/s}-1},\sqrt{(\lambda+K)^{4/s}-1}\right] \]
and we notice that the length of these intervals is increasing if $s<2$. Indeed,
	\[ \lim_{\lambda \to \infty} \sqrt{(\lambda+K)^{4/s}-1}-\sqrt{(\lambda-K)^{4/s}-1} = \lim_{\lambda \to \infty} \dfrac{\lambda^{4/s-1}}{\lambda^{2/s}} \]
which is $\infty$ if $s<2$. Thus, (\ref{eq:uncp}) holds for $\supp \hat u$ contained in any ball since (\ref{eq:uncp}) does not see modulation of $u$ (translation of $\hat u$).

We demonstrate that this is a violation of the uncertainty principle. Let $f(x)=\ind_{\{\gamma=0\}}(x) \phi(x)$, where $\phi$ is some positive $L^2$ function so that $f \in L^2$ and $\gamma f=0$. Then, $\hat f \in L^2$ so 
setting $g_R = \calf^{-1}(\ind_{B(0,R)} \hat f)$, $g_R$ converges to $f$ in the $L^2$ norm. Therefore, since $\supp \hat g_R \subset B(0,R)$, by (\ref{eq:uncp}), 
	\[ c\|g_R\| \le \|\gamma g_R\| \le \|\gamma f\| + \|\gamma(g_R-f)\| \le \|\gamma\|_\infty \|g_R-f\|. \]
The LHS goes to $c\|f\| > 0$ ($f$ is nonzero since $m(\{\gamma = 0\})>0$) while the RHS appoaches zero as $R \to \infty$ which is a contradiction.
\section*{Acknowledgements}
The author is thankful to Milena Stanislavova for introducing him to this problem as well as to Benjamin Jaye and Mishko Mitkovski for comments that greatly improved the presentation of this paper.

\bibliographystyle{plain}
\bibliography{/home/waton/mega/School/Research/refs-all/refs-all}

\begin{thebibliography}{10}

\bibitem{b-l-r}
C.~Bardos, G.~Lebeau, and J.~Rauch.
\newblock Sharp sufficient conditions for the observation, control, and
  stabilization of waves from the boundary.
\newblock {\em SIAM J. Control Optim.}, 30(5):1024--1065, 1992.

\bibitem{borichev10}
A.~Borichev and Y.~Tomilov.
\newblock Optimal polynomial decay of functions and operator semigroups.
\newblock {\em Mathematische Annalen}, 347(2):455--478, 2010.

\bibitem{burq16}
N.~Burq and R.~Joly.
\newblock Exponential decay for the damped wave equation in unbounded domains.
\newblock {\em Communications in Contemporary Mathematics}, 18(06):1650012,
  2016.

\bibitem{gearhart78}
L.~Gearhart.
\newblock Spectral theory for contraction semigroups on hilbert space.
\newblock {\em Transactions of the American Mathematical Society},
  236:385--394, 1978.

\bibitem{havin12}
V.~P.~Havin and B.~J{\"o}ricke.
\newblock {\em The uncertainty principle in harmonic analysis}, volume~28.
\newblock Springer Science \& Business Media, 2012.

\bibitem{huang85}
F.~Huang.
\newblock Characteristic conditions for exponential stability of linear
  dynamical systems in hilbert spaces.
\newblock {\em Ann. of Diff. Eqs.}, 1:43--56, 1985.

\bibitem{kovrijkine01}
O.~Kovrijkine.
\newblock Some results related to the {L}ogvinenko-{S}ereda theorem.
\newblock {\em Proceedings of the American Mathematical Society},
  129(10):3037--3047, 2001.

\bibitem{logvinenko74}
V.~N.~Logvinenko and J.~F.~Sereda.
\newblock Equivalent norms in spaces of entire functions of exponential type.
\newblock {\em Teor. Funkci{\i}Funkcional. Anal. i Prilozen. Vyp}, 20:102--111,
  1974.

\bibitem{malhi19}
S.~Malhi and M.~Stanislavova.
\newblock On the energy decay rates for the 1d damped fractional {K}lein-{G}ordon
  equation.
\newblock {\em arXiv preprint arXiv:1809.09531}, 2018.

\bibitem{malhi18}
S.~Malhi and M.~Stanislavova.
\newblock When is the energy of the 1d damped {K}lein-{G}ordon equation decaying?
\newblock {\em Mathematische Annalen}, 372(3-4):1459--1479, 2018.

\bibitem{paneah61}
B.~P.~Paneah.
\newblock Some theorems of {P}aley--{W}iener type.
\newblock In {\em Doklady Akademii Nauk}, volume 138, pages 47--50. Russian
  Academy of Sciences, 1961.

\bibitem{pazy83}
A.~Pazy.
\newblock {\em Semigroups of linear operators and applications to partial
  differential equations}, volume~44.
\newblock Springer-Verlag, 1983.

\bibitem{pruss84}
J.~Pr{\"u}ss.
\newblock On the spectrum of $C_0$-semigroups.
\newblock {\em Transactions of the American Mathematical Society},
  284(2):847--857, 1984.

\bibitem{rauch74}
J.~Rauch, M.~Taylor, and R.~Phillips.
\newblock Exponential decay of solutions to hyperbolic equations in bounded
  domains.
\newblock {\em Indiana university Mathematics journal}, 24(1):79--86, 1974.

\bibitem{zworski-book}
M.~Zworski.
\newblock {\em Semiclassical analysis}, volume 138.
\newblock American Mathematical Soc., 2012.

\end{thebibliography}
\end{document}